\documentclass[12pt]{amsart}
\usepackage{mathptmx}%times roman fonts, gives nice script O for Cuntz algebras
\usepackage{tikz}
\usetikzlibrary{positioning,arrows}
\usepackage{nicefrac}
\usepackage{varioref}
\usepackage{cool} %improved sum, int, diff commands AND provides the MeijerG-function that we will need
\usepackage{microtype} %packs text more, is supposed to make things look nicer
\usepackage{cite} %sorts multiple cites in nicer order.
\usepackage{mathtools}
\usepackage{suffix}%used in defining a Meijeg G-function
\usepackage{amsthm,amsmath}
\usepackage{hyperref} %popular package to add clickable links. Should be listed last in list of packages because it retcons many commands
\usepackage[capitalise]{cleveref}
\newtheorem{theorem}{Theorem}[section]
\newtheorem{lemma}[theorem]{Lemma}
\newtheorem{proposition}[theorem]{Proposition}
\newtheorem{corollary}[theorem]{Corollary}
\newtheorem{example}[theorem]{Example}
\newtheorem{definition}[theorem]{Definition} %
 
\DeclareMathOperator{\erf}{erf}
\DeclareMathOperator{\erfi}{erfi}

\newcommand{\F}{\mathcal F} %Fourier transform
\renewcommand{\Finv}{\F\inv} %inverse Fourier transform

\newcommand{\Id}{\text{Id}}
\newcommand{\inv}{^{-1}}

\newcommand{\KK}{{\textsl{KK}}} %Kasparov's KK-theory

\newcommand{\real}{\mathbb{R}}
\newcommand{\R}{\real}

\newcommand{\intSinc}{\int \sinc{xs}\,dS(s)}
\newcommand{\sinc}[1]{\frac{\sin{(#1)}}{{#1}}}
\newcommand{\mstack}[2]
{^{\displaystyle \mathstrut{#1}}_{\displaystyle \mathstrut{#2}}}	% could also be \mathstrut{\displaystyle ...

\begin{document}
	
	\title{\sc Shannon Integrals and applications to Riemann-Hilbert problems}%do title of article in small caps like in MSc thesis?
	\author{Y. Abdolmaleki and D. Kucerovsky}
	\email{dkucerov@unb.ca, yavar.abdolmaleki@unb.ca}
	\address{Dept. of Math.\\University of New Brunswick at Fredericton\\Fredericton, NB\\Canada  E3B 5A3}
	\begin{abstract}We say that a real function can be represented by a Shannon integral if it can be written in the form $$\intSinc$$ for some finite measure $dS.$ We develop a theory of Shannon integral representations for  positive definite functions, and show that, under mild conditions, $f=f^{+}+f^{-}$ with $f^{\pm}$
being positive definite solutions of an additive Riemann--Hilbert problem on
the real line; a multiplicative factorization $|f|=\Phi^{+}\Phi^{-}$ is
obtained for the Thorin--Bernstein class.
	\end{abstract}\maketitle
	
	\subsection{Summary and introduction}\label{sect:Introduction}

We say that a real function admits a Shannon integral representation if it can be written in the form
\[
f(x)=\int_{0}^{\infty}\frac{\sin(xs)}{xs}\,dS(s)\footnote{In some cases, when it has been clear, we have used $ds$ instead of $dS(s)$ or $dt$ instead of $dT(t).$}
\]
for some finite Borel measure \(S\) on \([0,\infty)\). We develop a systematic theory of such representations for positive definite functions. Under mild assumptions on \(S\), we show that for an even continuous function \(f:\mathbb{R}\to\mathbb{C}\) the following conditions are equivalent:
\begin{enumerate}
\item \(f\) has a Shannon integral representation with respect to \(S\);
\item \(f(x)=\int_{0}^{1}G(xt)\,dt\), where \(G\) is the Fourier--Stieltjes transform of \(S\);
\item \(G\) satisfies the differential relation \(G(x)=x f'(x)+f(x)\).
\end{enumerate}
If, in addition, \(S\) is positive, these are further equivalent to:
\begin{enumerate}
\setcounter{enumi}{3}
\item \(f\) is positive definite on \(\mathbb{R}\);
\item \(f\) is absolutely monotone on the positive imaginary axis;
\item the Fourier transform \(\widehat{f}\) is real, nonnegative, and monotone decreasing on \((0,\infty)\).
\end{enumerate}
In this case the additive Riemann--Hilbert problem for \(f\) on \(\mathbb{R}\) admits positive definite solutions
\[
f_{\pm}(z)=\int_{0}^{\infty}\frac{e^{\pm i s z}-1}{s z}\,dS(s).
\] The precise statements and proofs for the above are given in \cref{lem:de.form,cor:InverseFourierStieltjes,lem:pd.is.necessary,th:when.is.it.positive,th:additive.RH}. Thus, the Shannon integral representation property is equivalent to being a smoothed Fourier--Stieltjes transform, and also to a differential condition. If also the function is positive, a number of striking additional properties come into play (namely, monotonicity, and positive-definite RH factorizations).
We illustrate the theory with examples of explicit Shannon representations and positive definite Riemann--Hilbert factorizations for certain classical positive definite functions, including examples related to the Bhatia--Jain family \(x\mapsto (1-|x|)/(1-|x|^{\alpha})\). As an application to operator theory, we derive integral commutator formulae and norm estimates of L\"owner type for functions of self-adjoint operators, with particular emphasis on the Weyl algebra.

	Our main result on Riemann-Hilbert problems (\cref{th:additive.RH}) states that under certain conditions, a  function $f$ on the real line can be written in the form $f=\Phi_{+}+ \Phi_{-}$ with $\Phi_{+}$ and $\Phi_{-}$ are positive definite if and only if $f$ is positive definite. The multiplicative case (\cref{th:mult.RH}) is similar, utilizing the Thorin--Bernstein function class. Our main result on representation as a Shannon integral is probably 	\cref{cor:factorization}, which finds some equivalent conditions, under some assumptions on the measure, for a function $f(x)$ to have a representation as $f(x):=\intSinc.$
	
	As an  application of our theory of integral representations, we consider applications to operator commutators, and we illustrate using the special case of the Weyl algebra.
	Recall that the unital $*$-algebra $W$ with hermitian generators $p$ and $q$ satisfying the canonical commutation relations $[p,q]=i\Id$ is called the Weyl algebra\cite{Schmudgen,rosenberg2004}.
	The Weyl algebra inherits the usual Hilbert functional calculus for unbounded self-adjoint operators, (\emph{i.e.} a Weyl calculus) and  we may ask how $[f(p),q]$ and $[p,q]$ are related.
	This is addressed in \cref{cor:commutator}.

	\section{Definitions and Lemmas}

	\begin{definition}
	A function $G$ is called positive definite if for any system of real numbers  $  t _ {1} \dots t _ {n} $,
	$$
	\sum _ {i, j = 1 } ^ { n }
	G ( t _ {i} - t _ {j} )
	\xi_ {i} \overline \xi_{j}  \geq   0,
	$$
	whatever the complex numbers  $  \xi _ {1} \dots \xi _ {n} $. This property is preserved under pointwise limits $G_n\to G.$
	\end{definition}\label{KernelRef}
	
	\begin{definition}
	
	A function is called \emph{absolutely monotone} at a point $x$ if its derivatives at that point all exist and are non-negative. If the function is  absolutely monotone at all points on the non-negative real line, one simply says that the function is absolutely monotone. 
	\end{definition}
	
	Bernstein showed that such a function extends analytically to the complex plane. 
	The following Lemmas develop a theory of Shannon integral representations for  positive definite functions. We have typically, and where it has been clarifying, written $dM$, and sometimes $dm$, for a measure coming from Bochner-Khintchine, and $dS$ for the measure appearing in a Shannon integral.

	A function on the complex plane is said to be of exponential order  if there exist constants $M$ and $\alpha$ such that $|f(re^{i\theta})|\le Me^{\alpha r}.$
	\begin{lemma}\label{lem:tm.is.pd}
Let $f$ be analytic and even, and define $g(x):=f(ix)$, so that $g$ is real on
the real line. Then:
\begin{enumerate}
  \item $f$ is positive definite on the real line;
  \item $g(x)e^{Bx}$ is absolutely monotone on the positive real line, for
        some $B>0$.
\end{enumerate}
If $f$ is of exponential order, then (i) implies (ii). The converse,
(ii) implies (i), holds in general.
\end{lemma}
	\begin{proof}
		We show that $(i)$ implies $(ii)$ if $f$ is of exponential order. Given that the continuous function $f(x)$ is positive-definite on the real line, the Bochner-Khintchine theorem\cite{Lukacs1987} implies that $f(x) = \int e^{i\omega u} \, dM(\omega)$ for some non-negative Stieltjes measure $M.$ By the Paley-Wiener theorem\cite{schwartz}, the exponential order condition on $f$ means that the measure $M$ is supported in an  interval $[-B,B].$ Then $g(x)e^{B x}$ can be written $ \int e^{\omega x} \, dM'(\omega)$ where $M'$ is $M$ translated horizontally by an additive factor $B.$ The translated measure $M'$ is supported in $[0,2B],$ and by Bernstein-Widder's theorem\cite{bernstein}, $g(x)e^{B x}=\int e^{\omega x} \, dM'(\omega)$ is therefore absolutely monotone. For the converse, if $g(x)e^{Bx}$ is absolutely monotone on the positive real line, for some $B>0,$ then by Bernstein-Widder's theorem, $g(x)e^{Bx}=\int e^{sx}\,dM$ for some non-negative finite Borel measure $M.$ By absolute convergence, this equality extends to the complex plane, and so $f(x)=e^{-iBx}\int e^{isx}\,dM.$ This shows that $f$ is a limit of positive definite functions, hence is positive definite. 
	\end{proof}
	Now a function-theoretical lemma: 
	\begin{lemma}If $f$ is an entire function and has a Taylor series expansion at $x=0$ with all coefficients non-negative, then the derivatives $f^{[k]}(x)$ are non-negative on $[0,\infty).$ \label{lem:entire.is.tm}\end{lemma}
	\begin{proof} We are given  $f(x)=\sum a_n x^n$ with the $a_n$ being non-negative. Since the terms of the series are increasing functions on the positive real line, $[0,\infty),$ and the series converges there, it follows that $f$ is an increasing function on $[0,\infty).$ In particular, then, $f(x)\geq f(0) = a_0\geq0.$ Now, notice that the derivative $f'(x)$ is again an entire function and has a Taylor series expansion at $x=0$ with all coefficients non-negative. We proceed inductively, concluding that all the derivatives $f^{[k]}(x)$ are non-negative on $[0,\infty).$
	\end{proof}
	\begin{proposition} Let $f(x)=\intSinc$ be entire. The following are equivalent:
		\begin{enumerate}\item $f$ is absolutely monotone on the positive imaginary axis;
			\item $f$ is absolutely monotone at zero on the positive imaginary axis, and
			\item $f$ is positive definite on the real line.
		\end{enumerate}\label{lem:pd.is.necessary}
	\end{proposition}
	\begin{proof}  Clearly $(i)$ implies $(ii).$ We now show that $(ii)$ implies $(i).$
		A function $f$ of the form $\intSinc$ is even, and real on the imaginary axis, because the integrand $\sinc{xs}$ has these properties as a function of $x.$  Define $g(x)=f(ix)$ for convenience. We note that $g$ is an entire function that is real on the real line, because $f$ was an even function, and $g$ is given to be absolutely monotone at the point $x=0.$ So the Taylor series of $g$ at $x=0$ has non-negative coefficients, and we apply \cref{lem:entire.is.tm}. By this Lemma, the function $g$ is absolutely monotone on the positive real line, and that proves condition $(i).$  \Cref{lem:tm.is.pd}, direction $ii$ implies $i,$ shows that condition $(i)$ implies condition $(iii)$.
		We now show that $(iii)$ implies $(i).$  Since $f(x)$ is positive definite, by \cite[th.4.3]{buescu} , the entire extension of \(f\) is the
Fourier--Laplace transform of its Bochner measure throughout
\(\mathbb C\). Since \(f\) is even, that measure is symmetric.
Thus, for \(y\geq0\),
\[
 f(iy)=\int_{\mathbb R}\cosh(yt)\,d\mu(t),
\]
and differentiation under the integral shows that every derivative
with respect to \(y\) is nonnegative. Therefore the integral, $f,$ is also absolutely monotone on the positive imaginary axis.
	\end{proof}
	Next we characterize $f$ indirectly, through the Fourier-Stieltjes transform.
	
	We first recall the basic elements  of the Fourier-Stieltjes transform. Let the function  $  F $
	have bounded variation on  $  (- \infty , + \infty ) $.
	The function
	
	\begin{equation}
		G ( x)  = \
		{
			\frac{1}{\sqrt {2 \pi } }
		}
		\int\limits _ {- \infty } ^ { {+ }  \infty }
		e  ^ {- ixy}  dF ( y)
		\label{FStransform}\end{equation}
	
	is called the Fourier--Stieltjes transform of  $  F $.
	The function  $  G $
	determined by the integral \eqref{FStransform} is bounded and continuous.
	Formula \eqref{FStransform}  can be inverted: If  $  F $
	has bounded variation and if
	
	$$
	F ^  \circ   ( x)  = \
	{
		\lim_{\varepsilon\to0+}\frac{F ( x + \varepsilon) + F ( x - \varepsilon) }{2}
	} ,
	$$
	
	then
	
	\begin{equation}\label{eqn:InverseFourierStieltjes}
		F ^  \circ   ( x) - F ^  \circ   ( 0)  = \
		{
			\frac{1}{\sqrt {2 \pi } }
		}
		\int\limits_ {- \infty } ^ { {+ }  \infty }
		G ( \xi ) \frac{e^{i \xi x} - 1}{i \xi }
		d \xi,
	\end{equation}
	
	where the integral is taken to mean the Cauchy principal value at  $  \infty $. Although $F^{\circ}$ may differ from $F$ at points of discontinuity of $F,$ nonetheless it maps under the Fourier-Stieltjes transform to the desired function $G.$
	
An important fact for us  will be that, if one only allows non-decreasing functions of bounded variation as the function  $f$
	in formula \eqref{FStransform} , then the set of continuous functions  $G$
	thus obtained has a striking characterization, by the Bochner--Khinchin theorem, as the positive definite continuous functions. 
	The  Bochner–Khinchin theorem is a necessary and sufficient condition for a continuous function  $  G $
	to be positive definite, or equivalently, to be a positive multiple of  the  characteristic function of a probability distribution. This theorem is all the more remarkable because the continuous bounded functions in the range of the Fourier-Stietjes transform do not have any elegant characterization. 
	
	The Fourier transform of a finite measure $S$ is a continuous function, $G,$ and if the measure $S$ is in addition non-negative, then
	as discussed, the
	continuous function $G$ is positive definite on $\R.$ We now investigate the relationship between a function  $f(x)$ defined by $f(x)=\intSinc$ and $G(\omega),$ the Fourier transform $G=\F (S)$ of $S.$ The Fourier transform $G$ is a continous function on the real line, extending to an analytic function in the complex plane.

	\begin{lemma} \label{lem:odes} Let $f$ and $G$ be differentiable functions on the real line, with locally bounded derivatives.
		Then the following are equivalent:
		\begin{enumerate}\item $f(x):=\int_0^1 G(xt)\,dt$;
			\item $xf'+f=G(x).$
		\end{enumerate}
	\end{lemma}
	\begin{proof} Suppose first that equation $i$ is given. Note that the derivative $G'$ is bounded on any finite interval, so we may differentiate under the integral sign in $f(x):=\int_0^1 G(xt)\,dt,$ using a constant as the bound in the Lebesgue domininated convergence theorem. This gives
		\begin{align*}
			f'(x) &= \int_0^1 G'(xt) t \,dt \\
			&= \frac{G(x)}{x} - \frac{1}{x}\int_0^1 G(xt)\,dt. \\
			\intertext{Substitute $f(x):=\int_0^1 G(xt)\,dt$ to get }\\
			f'(x)&= \frac{G(x)}{x} - \frac{1}{x} f(x). \\
		\end{align*}  
		This shows that condition $i$ implies condition $ii$  for all nonzero $x.$ Since, at the point $x=0,$ conditions $i$ and $ii$ both  simplify to $f(0)=G(0),$ we have  that   condition $i$ implies condition $ii$ for all real $x.$
		For the converse direction, a similar argument shows that $f(x):=\int_0^1 G(xt)\,dt$ is a particular solution of the singular inhomogeneous differential equation $xf'+f=G(x).$ Adding a general solution for the related homogeneous differential equation $x f'+f=0$ we obtain a complete solution
		\[ f(x)= \int_0^1 G(xt)\,dt + C\frac1x .\] As we are seeking solutions in the class of continuous functions, we choose $C=0,$ thus recovering condition $i.$ This concludes the proof that condition $ii$ implies condition $i.$
	\end{proof}

	\begin{lemma} \label{cor:factorization} Let $S$ denote the Stieltjes measure of a bounded function, and denote its Fourier transform by $G(\omega).$
		Then the following are equivalent:
		\begin{enumerate}\item $f(x)=\intSinc$;
			\item $f(x)=\int_0^1 G(xt)\,dt$ and;
			\item $G(x)=x f'+f$.
		\end{enumerate}\label{lem:de.form}
	\end{lemma}
	\begin{proof} From the convolution property of the Fourier transform it follows that $(1)$ implies $(2).$ To see this, we first of all recall that  $$\sinc{\omega}=\frac12\int_{-\infty}^{\infty} \chi_{[-1,1]} e^{i\omega t}\,dt,$$ where $\chi_{[-1,1]}$ is the indicator function of the interval $[-1,1].$  Putting the above equation into equation $(1)$ and changing  the order of integration, as we may by Fubini's theorem because $S$ is a finite measure, we have  $f(x)=\frac12 \int \chi_{[-1,1]}G(xt)\,dt,$ where $G=\F (S).$ Since the integral kernel is an even function, the measure $S$ may as well be supposed to be symmetrical, its Fourier transform function $G(u)$ is real and even, and since the function $\chi_{[-1,1]}$ is also even, the integrand is therefore an even function of $t;$ so we can get rid of the factor of $\frac 12$, obtaining $f(x)=\int_0^\infty \chi_{[-1,1]} G(xt)\,dt$ as was to be shown. The converse direction is similar. 
		
		Since the measure $S$ is finite, the functions $G$ and $f$ are continuously differentiable and we can apply  \cref{lem:odes} to show the equivalence of $2$ and $3.$ 
	\end{proof}
	\begin{corollary} If $f(x)=\int \sinc{xs}\, dF(s)$ where $dF$ is the Riemann-Stieltjes measure of a  bounded function $F$ of bounded variation, then $F$ can be recovered from $f(x)$ by $$F^{\circ}(x)=\F\inv\!(xf'+f),$$ where $\F\inv$ denotes the inverse Fourier transform from \cref{eqn:InverseFourierStieltjes}. \label{cor:InverseFourierStieltjes}
	\end{corollary}
	%Even if one cannot verify the  condition on $F,$ the above may sometimes give a trial solution. 
	The above corollary completely determines when a function is in the class of interest, but requires computing a Fourier transform. We now give two sufficient conditions that do not involve Fourier transforms, and are stated entirely in terms of the function $f$ or sometimes $\frac{f(x)}x$. 
	
	The first one supposes convexity on the positive part of the real line, meaning to the right of the origin. 
	\begin{lemma}Let $f$ be an odd real function. Suppose that the derivative, $f',$ is a convex function on the positive part of the real line,  is continuous and goes to zero at infinity. Then there exists a finite positive Stieltjes measure $M$\footnote{As mentioned,  $dM$, and  $dm$, are used for measures coming from Bochner-Khintchine, and $dS$ for the measure appearing in a Shannon integral}  such that
		$\frac{f(x)}{x}= \int \frac{\sin sx}{sx} \,dM(s),$ for  all real $x.$ The same holds if instead of assuming that $f'$ is convex, we assume that $f'$ is positive definite.\label{lem:Polya1}\end{lemma}
	\begin{proof} Since $f'$ is an even real function whose graph is convex on the positive real line, P\'olya's criterion\cite{Lukacs1987} implies that it is positive definite. By Bochner-Khinchine, the even function $f'$ can then be written
		$$f'(x)=\int  \cos sx \,dM(s)$$ for some finite positive Stieltjes measure $M.$ Integrating from $0$ to $x$ and using Fubini's theorem to change order we conclude that
		$f(x)-f(0)=\int_0^\infty  \frac{\sin sx}{s} \,dM(s),$ giving a positive Stieltjes measure with the required properties. Since $f$ is odd, $f(0)=0.$ 
		For the last part, if we have assumed that $f'$ is already positive definite, then we can apply this fact directly instead of using P\'olya's criterion.
	\end{proof}
	
	For some important cases, such as $f$ with logarithmic behaviour
	at infinity, the following elementary result can be used:
	\begin{lemma}\ Let $f$ be a real and odd function. If there is a constant $
		C$ such that $f'-C$ and
		$f^{\prime\prime}$ are $L^2({\mathbb R}),$ then $\frac{f(x)}x =\int_{{\mathbb R}^{+}}\frac {\sin
			sx}{sx}\,dm(s),$ for some finite
		measure $S.$\label{lem:L1} \end{lemma}
	\begin{proof}\ Let $g=f'-C,$ and let the hat $\hat{\cdot}$ denote the usual
		exponential Fourier transform.
		By Parseval's theorem, $|\hat {g}(\omega )|,|\omega\hat {g}(\omega )|\in
		L^2({\mathbb R}).$ Let $M$ be the
		(measurable) set $M:=\{\omega\in (0,1):|\hat {g}(\omega )|\geq 1\}.$ 

                   Let $p(\omega ):=\max\{1,|\omega |^{-1}\}.$ Then
		\begin{eqnarray*}
			\int_{{\mathbb R}^{+}}|\hat {g}|&=&\int_M|\hat {g}|+\int_{I\setminus M}|\hat {
				g}|+\int_1^{\infty}|\hat {g}(\omega )\omega |\cdot |p(\omega )|\,d\!\omega\\
			&\leq&\|\hat {g}\|_2^2+1+\|\hat {g}(\omega )\omega\|_2\cdot\|p\|_2\end{eqnarray*}
		We see that $\hat {g}$ is $L^1.$ Thus
		\[f'(x)=\int_0^{\infty}\cos sx\cdot Cd\delta_0(s)+\int_0^{\infty}\cos sx\cdot
		2\hat {g}(s)\,d\!\mu (s),\]
		where $\mu$ is the standard Lebesgue measure, and $\delta_0$ is the
		point measure, or mass measure, with support at 0. As $f$ is an odd function, $f(0)=0$, and so the las equation reduces to 
                   \[f'(x)=C+2\int_0^{\infty}\cos sx\cdot
		\hat {g}(s)\,d\!\mu (s)\]
Integratinging both sides of this equality with respect to $x$ and using first fundamental theorem we get 
                  \[f(x)=Cx+2\int_0^x\int_0^{\infty}\cos sx\cdot
	        \hat {g}(s)\,d\!\mu (s)\,d t.\]
We showed above that $\hat{g}\in L^1(\mathbb{R}),$ and we have 
                  \[\int_0^x\int_0^{\infty}|\cos sx\cdot
	        \hat {g}(s)|\,d\!\mu (s)\,d t \leq \int_0^x\int_0^{\infty}|\
	        \hat {g}(s)|\,d\!\mu (s)\,d t= x\int_0^{\infty}|\hat{g}(s)|<\infty.\]
	        
So by Fubini's theorem we can change the order of integration and calculate $f(x)$ explicilty. This allows us to write down $f(x)$ explecitly as \[f(x)=Cx+2\int_0^{\infty}\frac{\sin sx}{s}\hat{g}(s)ds\cdot\]After deviding both sides by $x$, it suffices to define 
\[dm(\omega):=Cd\delta_0(\omega) + \hat{g}d\omega \footnote{Of course $\omega$ is just a dummy variable but instead of s, here, we use $\omega $ to be consistent with Fourier variable}\]

\end{proof}
	\begin{theorem} Let $f(x)=\intSinc,$ with $S$ a finite signed measure. Then the first two are equivalent, and imply the third:
		\begin{enumerate}\item The Fourier transform of $f$ is real and monotone decreasing on $[0,\infty);$
			\item $S$ is a positive finite measure, and
			\item $f$ is positive definite.
		\end{enumerate}\label{th:when.is.it.positive}
	\end{theorem}\begin{proof}Since the given integral kernel $\sinc{xs}$ is an even function, we can always suppose that the measure $S$ is supported within $[0,\infty).$ Noting that the given function $f$ is in $L^2(R)$ we compute its inverse $L^2$-Fourier transform, \begin{equation}\check f =\int \frac{1}{2s}\chi_{[-s,s]}(x)\,dS(s),\label{eq:change.here}\end{equation} 
		where $\chi$ denotes an indicator function and we have implicitly used Fubini's theorem to change order of integration.  
		Now we observe that because the indicator function $\chi_{[-s,s]}(x)$ is zero when $s$ is less than $x,$ we can restrict the domain of integration, to $s\geq x>0,$ obtaining
		\begin{equation}\check f(x) =\int^{\infty}_x \frac{1}{2s}\,dS(s).\label{eq:RadonNikodym}\end{equation}
		It is a key point that the integrand is a positive function, and so if $S$ is a positive measure, then the above function increases as $x$ decreases. This shows that property $(ii)$ implies property $(i).$ Furthermore, if the measure $S$ is not in fact positive, then the Hahn decomposition gives us a set on which the measure is negative, and where the above function is thus not monotone decreasing. Thus the properties $(i)$ and $(ii)$ are equivalent. On the other hand, if property $(i)$ holds, then the Fourier transform of $f$ is an even and real function that is monotone decreasing on  $[0,\infty);$ and is in $L^2(R),$ so that it necessarily goes to zero at infinity. A function that does all this must be a non-negative function, and so this shows that the function $F$ is positive definite. This shows that $i$ implies $iii.$ 
	\end{proof}
	
	\begin{corollary} If $f(x)=\intSinc,$ with $S$ being absolutely continuous against Lebesgue measure, then 
		\[ -2s\frac{d}{dx} \Finv(f) = \left[\frac{dS}{d\mu}\right]. \]
	\end{corollary}
	%I would like to claim that this shows some sort of L^2 density of our functions in L^2. however, there seem to be three different topologies around:
	%the weak topology on measures, the bounded variation norm on measures, and the L^2 norm. I think the images of absolutely continouous measures that are both L1 and L2 are dense in L2(R) 
	\begin{proof} Differentiate both sides of \cref{eq:RadonNikodym} with respect to $x.$  As $S$ is absolutely continuous with respect to Lebesgue measure, it has a (Radon-Nikod\'ym) density $\left[\frac{dS}{d\mu}\right],$   and differentiating the lower boundary on the right hand side of  \cref{eq:RadonNikodym} with respect to $x$ gives a factor of $-\left[\frac{dS}{d\mu}\right],$ see \cite[pg. 303 Ex. 17b]{royden}, obtaining in total  \(-2s\frac{d}{dx} \Finv(f)(x) = \left[\frac{dS}{d\mu}\right]. \)  \end{proof}
	\begin{example} This is a counter-example. We can represent a positive definite function $f,$ $$f(x)=\intSinc,$$ as a Shannon integral against a measure $S$ that is not positive. This shows that  \cref{th:when.is.it.positive} cannot be improved to include the statement that $(iii)$ implies $(i).$
		Let the measure $S$ be a sum of point measures, $S=\delta(x-3)+\delta(x-2)-\delta(x-\frac32 ).$ 
		Clearly $f(x)=\frac{\sin \! \left(2 x \right)}{2 x}+\frac{\sin \! \left(3 x \right)}{3 x}-\frac{ \sin \! \left(\frac{3 x}{2}\right)}{\frac{3 x}{2}}.$
		This function is in $L^2(\R)$ and if we check positive definiteness by taking the inverse $L^2$-Fourier transform, we get the non-negative function
\renewcommand{\arraystretch}{1.17}%widen rows a bit
		\[\check f(\omega)=\left\{\begin{array}{ccl}
			\frac{1}{12}&\mbox{ if } & |\omega| \le \frac{3}{2} 
			\\
			\frac{5}{12}&\mbox{ if } &\frac{3}{2}< |\omega| \le 2 
			\\
			\frac{1}{6}&\mbox{ if } &2< |\omega| \le 3 
			\\
			0&\mbox{ if } & 3<|\omega|  
		\end{array}\right.
		.\]
	\end{example}
Now we explain this through a complete characterization that will show that positivity of the measure $S$ is a sufficient but not necessary condition for positive definiteness. In order to shorten the statements, we suppose that the measure $S$ is finite and supported on $[0,\infty)$, as was several times assumed already.
 \begin{lemma}\label{lem:pdf.f}
Let $f(x)=\intSinc$ with $S$ a finite measure on $[0,\infty).$ Then $f$ is positive definite on $\mathbb{R}$ if and
only if
\[
  \int_{[\,\lvert x\rvert,\,\infty)}\frac{dS(s)}{2s}\;\ge\;0
  \qquad\text{for almost every } x\in\mathbb{R}.
\] 
\end{lemma}

\begin{proof}
The proof of \cref{th:when.is.it.positive}, specifically \vref{eq:RadonNikodym}, states that the inverse $L^{2}$-Fourier transform of $f$ is the left expression above. By the Bochner--Khinchin theorem $f$ is positive
definite exactly when that expression  is non--negative \emph{a.e.}
\end{proof}For completeness, we mention that such integral inequalities have been considered also in \cite[Prop. 3.10]{csordas}.
	\section{Riemann-Hilbert problems}

	Suppose that we wish to solve the additive Riemann-Hilbert problem for an even function $f(r)$ on the real line. This means we want a function
	$f^+ (z)$ that is analytic and nonsingular everywhere in the upper half plane and a function $f^{-}(z)$ that is  analytic and nonsingular everywhere in the lower half plane, with $$f(r)=f^{+}(r)+f^{-}(r)$$ for all real $r.$ For an introduction to the Riemann-Hilbert theory, see \cite{AblowitzFokas1997}.
	
	\begin{theorem}[Additive Riemann-Hilbert problem]\label{th:additive.RH} Suppose that $f(x)=\intSinc$ for some finite complex  measure $S$ on $[0,\infty).$ Then 
		\begin{equation}f^{+}(z)=\int \frac{e^{isz}-1}{2isz}\,dS,\quad\mbox{and}\quad f^{-}(z)=\int \frac{1-e^{-isz}}{2isz}\,dS\label{eq:decomp.SincInt}\end{equation}
		solve the additive Riemann-Hilbert problem for $f.$ If $f$ is positive definite up to a constant, then so are $f^{+}$ and $f^{-}.$ Moreover, $f^{+}(z)=f^{-}(-z).$
	\end{theorem}
	\begin{proof}
		Notice that the algebraic identity
		\begin{equation} f_0(x)=\sinc{sz}= \underbrace{\left\{\frac{e^{isz}-1}{2isz} \right\}}_{f^{+}_0} + 
                            \underbrace{\left\{ \frac{1-e^{-isz}}{2isz}\right\}}_{f^{-}_0}\label{eq:decomp}\end{equation}
		has the property that the first bracketed term is bounded and analytic in the upper half plane, while the second bracketed term is bounded and analytic in the lower half-plane. 
		Integrating both sides of \cref{eq:decomp} against the given finite complex measure $S$ preserves both boundedness in a half-plane and analyticity, giving the decomposition of \cref{eq:decomp.SincInt}, as was to be shown. Noting that the measure is supported on $[0,\infty)$,   identity $f^{+}(z)=f^{-}(-z)$ is straightforward to verify. 

This leaves the statements on positive-definiteness. The Riemann-Hilbert problem is trivial for the case of a constant, and linearity  allows us to subtract the constant, if there is one. This reduces to the case of  positive definite $f.$ 

The function $f(x)$ is, under our hypotheses, positive definite if and only if \[
  \int_{[\,\lvert x\rvert,\,\infty)}\frac{dS(s)}{2s}\;\ge\;0
  \qquad\text{for almost every } x\in\mathbb{R}.
\]
This is because of \cref{lem:pdf.f}. Similarly, the function $f^{+}(x)$ is, under our hypotheses, positive definite if and only if \[
  \int_{[\, x,\,\infty)}\frac{dS(s)}{2s}\;\ge\;0
  \qquad\text{for almost every } x\geq0.
\]
This comes from modifying the proof of \cref{lem:pdf.f} using the fact that $\int_{0}^1 e^{iusz}\,du=\frac{e^{isz}-1}{isz}.$ %(See \cref{eq:change.here}.)
 Thus, the positive definiteness of $f$ implies that of $f^{+}$ and hence of $f^{-}.$
	\end{proof}

	\begin{lemma}Let $f$ be an even real function. Suppose that the logarithmic derivative $\nicefrac{f'}{f}$ is a convex function that is continuous and goes to zero at infinity. Then we can solve the multiplicative Riemann-Hilbert problem for $f.$ \end{lemma}
	\begin{proof} Rescaling the given function $f$ by a constant, we may as well assume that $f(0)=1.$ Lemma \ref{lem:Polya1} shows there exists a finite positive Stieltjes measure $M$ such that
		$\frac{\ln |f(x)|}{x}=\int_0^\infty  \frac{\sin sx}{sx} \,dM(s),$  for  all real $x.$  \Cref{th:additive.RH} solves the additive Riemann--Hilbert problem, giving  $\frac{\ln |f(x)|}{x}=f^{+}(x)+f^{-}(x).$ Multiplying by the entire function $x$ and  exponentating we get a Riemann-Hilbert factorization $f(x)=\exp(xf^{+}(x))\exp(xf^{-}(x)) .$  This works because if a function is analytic and without singularities in, say, the upper half plane, then its exponential is also analytic and without singularities in the upper half plane.
	\end{proof}
	
	% Exponentiating both sides we get solutions of a multiplicative Riemann-Hilbert problem.
	To improve the lemma, we recall the Thorin-Bernstein classes.
	The Thorin-Bernstein class  is one of several possible ways to define a cone of functions that is closed under pointwise limits and has a useful monotone property.  They are defined as complete Bernstein functions whose derivative is a Stieltjes function  \cite[pg.109]{schilling2012}. Bernsteins's little theorem provides extensions to analytic functions, and we may as well assume that this has been done.
	
	For example, the functions $x^\alpha$ and $\ln(1+x^\alpha)$ are in the Thorin-Bernstein class for $0<\alpha<1.$ 
	\begin{theorem} Let the even analytic function $f$ be in the Thorin-Bernstein class on the positive imaginary axis.
		Then there exist solutions to the multiplicative Riemann-Hilbert problem for $|f|.$\label{th:mult.RH}
	\end{theorem}
	\begin{proof} 
		We recall the facts that a not identically zero function $f$ is a complete Bernstein function if and only if $\nicefrac{1}{f}$ is a Stieltjes function \cite[Th. 7.3]{schilling2012}, and that Stieltjes functions are a subclass of the completely monotone functions. Thus, $f$ being Thorin-Bernstein implies that $\nicefrac{1}{f}$ is a Stieltjes function, hence completely monotone. Furthermore, $f'$ is again a Stieltjes function, hence completely monotone. Thus $f'$ and $\nicefrac{1}{f}$ are both completely monotone on the positive imaginary axis, so that their product $\nicefrac{f'}{f}$ is also completely monotone there \cite[p.5]{schilling2012}. But then if we replace the independent variable $z$ by $-z$ we obtain absolute monotonicity, and by \cref{lem:tm.is.pd} we then have positive definiteness on the real line. This representation persists if we change back to the original variable, and so we have  the representation
                    
		$$(\nicefrac{f'}{f})(z) = \int  e^{-izs}\,dM(s). $$
		Integrating this bounded function with respect to $z$ from $0$ to $x,$ we have
		$$\ln f = \ln f(0)+i\int \frac{e^{-izs}}{s} \,dM(s). $$
		Taking real parts we get the logarithm of the modulus of $f$
		$$\ln |f| = \ln |f(0)|+\int \frac{\sin{zs}}{s} \,dM(s). $$
		This provides the needed integral representation, and therefore as in \cref{th:additive.RH} get solutions 
                   \begin{align*}
                   \ln |f^{+}(z)| &=\frac12                     \ln |f(0)|+\int\frac{e^{isz}-1}{s}\,dM(s),\quad\mbox{and} \\
                   \ln |f^{-}(z)| & =\frac12 \ln |f(0)|+\int\frac{1-e^{-isz}}{s}\,dM(s),
                   \end{align*}
 giving a factorization.

	\end{proof}%Actually, we already showed that the integrands are positive definite. So the RHSs are positive definite, and since Taylor series with positive coefficients take pd to pd, the solutions are pd!
	\section{Examples}
	
	In this section we apply our theory to specific cases. This shows the effectiveness of our framework for finding positive definite functions. 
	Mixtures of Gaussians are common in statistics, so it is useful to solve related Riemann-Hilbert problems:
\begin{theorem}
Given a Riemann--Hilbert problem where $f$ is presented as a Weierstrass
transform, $f(x)=\int e^{-(x-y)^{2}}\,dM(y)$ with $M$ a finite measure, the
functions
                   \begin{align*}
                   f^{+}(z) &=\frac{i}{2}\int\big(\erf(i(z-y))+1\big)e^{-(z-y)^{2}}\,dM(y),\quad\mbox{and} \\
                   f^{-}(z) &=\frac{i}{2}\int\big(\erf(-i(z-y))+1\big)e^{-(z-y)^{2}}\,dM(y)
                  \end{align*}

solve the additive Riemann--Hilbert problem for $f$.
\end{theorem}

\begin{proof}
Take first the case $f(x)=e^{-x^{2}}$. Taking the inverse Fourier transform of
the expression given by Lemma~\ref{lem:de.form}\,(iii) gives the explicit
Shannon representation
\[
  e^{-x^{2}}=\frac{1}{2\sqrt{\pi}}\int_{0}^{\infty}
     s^{2}e^{-s^{2}/4}\,\frac{\sin(sx)}{sx}\,dS(s) ,
\]
so that, by Eq.~\eqref{eq:decomp.SincInt},
\begin{align*}
  f^{+}(z)=\frac{1}{2\sqrt{\pi}}\int_{0}^{\infty}
     s^{2}e^{-s^{2}/4}\,\frac{e^{isz}-1}{2isz}\,dS(s)
   =&\tfrac12 e^{-z^{2}}\big(1+i\,\erfi(z)\big)\\
   =&\tfrac{i}{2}\big(\erf(iz)+1\big)e^{-z^{2}}.
\end{align*}

Likewise $f^{-}(z)=\tfrac{i}{2}\big(\erf(-iz)+1\big)e^{-z^{2}}$.
As a check, $\Re f^{+}(x)=\tfrac12 e^{-x^{2}}\footnote{$\Re $ shows the real part.}$ for real $x$, and
$f^{+}+f^{-}=f$, as required. We then integrate these solutions against the
given measure $M$.
\end{proof}

	Now let us next consider the frequently occurring function $f(x)=x^\alpha,$ the power function. We wish to write this in the form
	$$x^\alpha = \intSinc. $$ The table of Fourier transforms in \cite{Milton} suggests that doing so would take us outside the framework of finite measures that we have constructed, but we may proceed formally and verify the result, if one is obtained.
	Formally, applying   \cref{lem:de.form} part $(3)$ we have that the Fourier transform of the desired (symmetric positive) measure $dS,$ if it exists, satisfies $\F(dS)=(1+\alpha) x^\alpha.$   This helped us to find a proof of the following, where the integral is understood as an improper (conditionally convergent)
integral; note that the representing measure $s^{-\alpha}\,dS$ is not finite,
so this example lies outside the finite-measure framework of
Section~\ref{sect:Introduction}.
\begin{proposition}\label{lem:xalpha}
Let $0<\alpha<1$. For $x>0$,
\[
  x^{\alpha}
  \;=\;\frac{\alpha}{\Gamma(1-\alpha)\,\sin(\pi\alpha/2)}
  \int_{0}^{\infty}\frac{\sin(xs)}{s}\,s^{-\alpha}\,dS(s).
\]

\end{proposition}

\begin{proof}
Substituting $t=xs$ gives
$\int_{0}^{\infty}\sin(xs)\,s^{-1-\alpha}\,dS(s)
 = x^{\alpha}\int_{0}^{\infty}t^{-1-\alpha}\sin t\,dT(t)$.
Integration by parts and \cite[p.\,77]{Temme} give
\[
  \int_{0}^{\infty}t^{-1-\alpha}\sin t\,dT(t)
  =\frac{1}{\alpha}\int_{0}^{\infty}t^{-\alpha}\cos t\,dT(t)
  =\frac{\Gamma(1-\alpha)}{\alpha}\sin\!\Big(\frac{\pi\alpha}{2}\Big).
\]
Since $0<\alpha<1$, both integrals converge at $0$ and at $\infty$ in the
improper sense. Solving for $x^{\alpha}$ gives the stated constant.
\end{proof}

	The next problem we consider is the Bhatia-Jain function $f(x)=\frac{1-|x|}{1-|x|^{\alpha}}.$ This function was introduced as a  challenge problem in \cite{BhatiaJain} and further studied in \cite{Berg2017}. Thus we would like to see if our theory can be applied to a deliberately challenging case. Theorem 1.3 in  \cite{BhatiaJain} shows that $f(x)$ is unexpectedly positive definite for many rational values of $\alpha;$ we take here a typical case, namely  $\alpha=3/2.$ \Cref{th:when.is.it.positive} showed that being positive definite is a necessary condition for being representable as a Shannon integral of some positive measure. We may therefore hope to show that
	$$f(x)=\frac{1-|x|}{1-|x|^{3/2}}= \intSinc$$ for some positive finite measure $S.$
	
	Applying   \cref{lem:de.form} part $(3)$ we have that the Fourier cosine transform of the desired measure $dS,$ if it exists, satisfies $$ \F(dS)=xf'(x)+f(x)=\frac{- 4 x+  x^{\frac{5}{2}}+  x^{\frac{3}{2}}+ 2}{2\left(- 1+x^{\frac{3}{2}}\right)^{2}}.$$
	Thus, we wish to understand the inverse Fourier cosine transform of
	$$\frac{- 4 x+  x^{\frac{5}{2}}+  x^{\frac{3}{2}}+ 2}{2\left(- 1+x^{\frac{3}{2}}\right)^{2}},$$ if it exists. Throughout this computation we use the cosine-transform convention
$G(t)=\int_{0}^{\infty}g(s)\cos(ts)\,ds$, equivalently
$g(s)=\frac{2}{\pi}\int_{0}^{\infty}G(t)\cos(st)\,dt$, and we write $b=s$ for
the transform variable. Note that $t=1$ is a removable singularity of $G$,
with $G(1)=1/2$, and that $G(0)=f(0)=1$. We will need to use generalized hypergeometric functions. See \cite{BealsSzmigielski2017} for more information on hypergeometric functions.
	To take the inverse cosine transform of a rational function of $x^{1/2},$ the process we follow is to substitute $x=w^2$ into the inverse transform integral,   form a partial fraction expansion of the rational function, and then notice that each of the terms in the expansion is of a generalized Fresnel integral form similar to, for example,
	
	$$\int_0^{\infty}\frac{\cos (bw^2)}{aw+1}dw=\frac{\sqrt{2}\, \left(b^{2}\right)^{\frac{3}{4}} G_{5,3}^{3,5}\! \left(\frac{4 a^{4}}{b^{2}}\left| {\mstack{0,\frac{1}{4},\frac{1}{2},\frac{3}{4},\frac{3}{4}}{\frac{3}{4},\frac{1}{2},0}}\right.\right)}{8 \pi^{\frac{5}{2}} b^{2}}.$$
	This process is not suitable for hand computation, rather, the point is to show that our theory can provide explicit Shannon integral representations even in challenging cases.
	Using a symbolic integration program (Maple) to perform this process and simplify the result, we get that the desired  measure, specificially, $dS$ is Lebesgue measure weighted by the function%maple defines mstack, a function of two variables, in maple.sty
	\begin{multline*}
		g(s):=\frac{2 \sqrt{3}\, G_{3,9}^{7,3}\! \left(\frac{b^{6}}{46656}\left| {\mstack{\frac{1}{4},\frac{1}{2},\frac{3}{4}}{\frac{5}{4},1,\frac{5}{6},\frac{3}{4},\frac{1}{2},\frac{1}{2},\frac{1}{6},\frac{2}{3},\frac{1}{3}}}\right.\right)}{\pi^{\frac{5}{2}} {b}}\\
		-\frac{24\sqrt{3}\, G_{3,9}^{7,3}\! \left(\frac{b^{6}}{46656}\left| {\mstack{\frac{1}{4},\frac{3}{4},1}{\frac{5}{4},1,1,\frac{3}{4},\frac{2}{3},\frac{1}{2},\frac{1}{3},\frac{7}{6},\frac{5}{6}}}\right.\right) }{\pi^{\frac{5}{2}} b^{2}}\\
		-\frac{18 \sqrt{2}\, G_{4,10}^{7,4}\! \left(\frac{b^{6}}{46656}\left| {\mstack{\frac{1}{4},\frac{1}{2},\frac{3}{4},1}{\frac{5}{4},\frac{13}{12},1,\frac{3}{4},\frac{3}{4},\frac{1}{2},\frac{5}{12},\frac{5}{4},\frac{11}{12},\frac{7}{12}}}\right.\right)}{\pi^{\frac{5}{2}} b^{\frac{5}{2}}}\\
		+\frac{108 \sqrt{2}\, G_{3,9}^{7,3}\! \left(\frac{b^{6}}{46656}\left| {\mstack{\frac{1}{4},\frac{1}{2},1}{\frac{5}{4},\frac{5}{4},1,\frac{11}{12},\frac{3}{4},\frac{7}{12},\frac{1}{2},\frac{17}{12},\frac{13}{12}}}\right.\right)}{\pi^{\frac{5}{2}} b^{\frac{7}{2}}}\\
	\end{multline*} Plotting this, it appears to be positive as expected.
	Thus $$f(x)=\frac{1-|x|}{1-|x|^{3/2}}= \int \sinc{xs} g(s)\,dS(s),$$ with $g(s)$ as above. It then follows  that we can explicitly solve the additive Riemann-Hilbert problem for this function by \cref{th:additive.RH}, obtaining $$f^{+}(z)=\int_0^\infty \frac{e^{isz}-1}{sz}g(s)\,dS(s),\quad\mbox{and}\quad f^{-}(z)=\int_0^\infty \frac{1-e^{-isz}}{sz}g(s)\,dS(s).$$
	\medskip
	
	The method presented depends on the rationality of $\alpha.$ It is plausible that methods based on contours and Cauchy's theorem as  in \cite{Berg2017} could also be used.
	\section{Commutators and anti-commutators}
Operator commutators measure the failure of two operations to commute, thereby revealing structural and dynamical information not visible from the operators separately.  
For fixed \(A\), the map \(\operatorname{ad}_A(B)=[A,B]\) satisfies
\[
[A,BC]=[A,B]C+B[A,C],
\]
so it is an inner derivation, or a noncommutative analogue of differentiation.  
Consequently, commutators play a central role in operator algebras, geometry, and mathematical physics, where they encode infinitesimal symmetries and evolution.\cite{Connes}

The following Duhamel-type identity gives an application of Shannon integral
representations to the functional calculus of self-adjoint operators.

\begin{proposition}\label{prop:Duhamel}
Suppose that
\[
  \frac{f(x)}{x}=\int_{[0,\infty)}\frac{\sin(sx)}{sx}\,dm(s),
\]
where $m$ is a finite complex Borel measure on $[0,\infty)$. If $X$ and $Y$
are bounded self-adjoint operators, then
\begin{align}
f(X)-f(Y)
  =\int_{[0,\infty)}\frac1s\int_0^s
  \bigl(&\cos(tX)(X-Y)\cos((s-t)Y) \notag\\
       &-\sin(tX)(X-Y)\sin((s-t)Y)\bigr)
  \,dT(t)\,dm(s).                                      \label{eq:Duhamel}
\end{align}
The operator-valued integral converges in norm, and
\[
  \|f(X)-f(Y)\|
  \leq \|X-Y\|\,\|m\|_{\mathrm{TV}}.
\footnote[1]{$||.||_{TV}$ stands for norm in total variation. }\] 
\end{proposition}

\begin{proof}
For fixed $s\geq0$, differentiating
\[
  t\longmapsto e^{itX}e^{i(s-t)Y}
\]
in operator norm gives
\[
  \frac{d}{dt}\bigl(e^{itX}e^{i(s-t)Y}\bigr)
  =i e^{itX}(X-Y)e^{i(s-t)Y}.
\]
The fundamental theorem of calculus therefore gives
\[
  e^{isX}-e^{isY}
  =i\int_0^s e^{itX}(X-Y)e^{i(s-t)Y}\,dT(t).
\]
Taking imaginary parts yields
\begin{align*}
\sin(sX)-\sin(sY)
  =\int_0^s
  \bigl(&\cos(tX)(X-Y)\cos((s-t)Y)\\
       &-\sin(tX)(X-Y)\sin((s-t)Y)\bigr)\,dt.
\end{align*}
Integration against $dm(s)/s$ gives \eqref{eq:Duhamel}. Moreover,
\[
  \|\sin(sX)-\sin(sY)\|
  \leq s\|X-Y\|,
\]
and hence
\[
  \|f(X)-f(Y)\|
  \leq\int_{[0,\infty)}
      \frac{\|\sin(sX)-\sin(sY)\|}{s}\,d|m|(s)
  \leq\|X-Y\|\,\|m\|_{\mathrm{TV}}.
\]
\end{proof}

This estimate is useful in connection with operator differentiability, an
application of the Fréchet derivative to the functional calculus. For more on
this connection, see \cite{flett,Pedersen2000}.

	\begin{corollary}\label{cor:commutator}
Let $D$ be self-adjoint and let $a$ be a bounded operator such that
$a\operatorname{Dom}(D)\subseteq\operatorname{Dom}(D)$ and the commutator
$[D,a]$, initially defined on $\operatorname{Dom}(D)$, extends to a bounded
operator. Then
\begin{align}
[f(D),a]
  =\int_{[0,\infty)}\frac1s\int_0^s
  \bigl(&\cos(tD)[D,a]\cos((s-t)D) \notag\\
       &-\sin(tD)[D,a]\sin((s-t)D)\bigr)
  \,dT(t)\,dm(s),                                      \label{eq:commutator}
\end{align}
where the integral converges in operator norm. In particular,
\[
  \|[f(D),a]\|
  \leq\|[D,a]\|\,\|m\|_{\mathrm{TV}}.
\]
\end{corollary}

\begin{proof}
Under the stated hypotheses,
\[
  [e^{isD},a]
  =i\int_0^s e^{itD}[D,a]e^{i(s-t)D}\,dT(t).
\]
Taking imaginary parts and integrating against $dm(s)/s$ gives
\eqref{eq:commutator}. The norm estimate follows in the same way as in
Proposition~\ref{prop:Duhamel}.
\end{proof}
	
	Now let us consider representations of the Weyl algebra, generated by operators satisfing $[x,p]=i \Id.$ Using \cref{lem:xalpha} we have:
\begin{corollary}
Let $0<\alpha<1$. Subject to the usual domain assumptions, if $x$ is positive
and $[x,p]=i\,\mathrm{Id}$, then $[x^{\alpha},p]=i\alpha x^{\alpha-1}$.
\end{corollary}

\begin{proof}
The measure $s^{-\alpha}\,dS$ of Proposition~\ref{lem:xalpha} is not of finite
total variation, so Proposition~\ref{prop:Duhamel} does not apply directly.
We instead truncate: for $0<\varepsilon<R<\infty$ let
$m_{\varepsilon,R}:=\chi_{[\varepsilon,R]}\,s^{-\alpha}\,dS$, a finite measure,
and let $x_{\varepsilon,R}^{\alpha}$ denote the corresponding truncated
integral. Corollary~\ref{cor:commutator} applies to each truncation, giving
\[
  [x^{\alpha}_{\varepsilon,R},p]
  = i\,\frac{\alpha}{\Gamma(1-\alpha)\sin(\pi\alpha/2)}
    \int_{\varepsilon}^{R} s^{-\alpha}\cos(xs)\,dS(s) .
\]
Both sides converge as $\varepsilon\to0^{+}$, $R\to\infty$ in the strong
topology (on a dense domain), and by
\cite[p.\,77]{Temme},
$\int_{0}^{\infty}s^{-\alpha}\cos(xs)\,dS(s)
 = x^{\alpha-1}\Gamma(1-\alpha)\sin(\pi\alpha/2)$.
The two occurrences of $\Gamma(1-\alpha)\sin(\pi\alpha/2)$ cancel, leaving
$[x^{\alpha},p]=i\alpha x^{\alpha-1}$.
\end{proof}
	\section*{Thanks and Statements}
	We  thank NSERC (Canada) for financial support. We state no conflicts of interest. Data sharing is not applicable to this article as no datasets were generated or analysed during the current study.
	
\end{document}